\documentclass[10pt]{amsart}
\usepackage{amsmath,amssymb,verbatim}
\usepackage[utf8]{inputenc}
\usepackage{titlesec}
\titleformat{\section}[hang]
{\upshape\bfseries}{\thesection.}{5pt}
{\large\bfseries}
\titleformat{\subsection}[hang]
{\itshape}{\thesubsection.}{5pt}
{}

\usepackage{enumerate}
\usepackage{pdflscape}
\usepackage{amsthm}
\usepackage{mathrsfs}
\usepackage{color}
\usepackage[normalem]{ulem}
\usepackage{cancel}
\usepackage{tikz-cd}
\usepackage{tikz}
\usetikzlibrary{matrix}
\usepackage[all]{xy}
\usepackage{mathtools}
\usepackage{cleveref}
\usepackage{arydshln}
\usepackage{enumitem}
\usepackage{amsaddr}

\topskip=\baselineskip 
\newtheorem{theorem}{Theorem}[section]
\newtheorem{lemma}[theorem]{Lemma}
\newtheorem{proposition}[theorem]{Proposition}

\newtheorem{remark}[theorem]{Remark}

\newtheorem{maintheorem}{Theorem}
\newtheorem{maincorollary}[maintheorem]{Corollary}

\makeatletter
\@addtoreset{equation}{section}
\makeatother

\newcommand{\Gal}{{\rm Gal}}

\newcommand{\Aut}{\mbox{\rm Aut}}

\newcommand{\Imagen}{\mbox{\rm Im }}

\newcommand{\N}{{\mathbb N}}
\newcommand{\Z}{{\mathbb Z}}
\newcommand{\Q}{{\mathbb Q}}

\newcommand{\C}{{\mathbb C}}

\newcommand{\HQ}{{\mathbb H}}

\newcommand{\GEN}[1]{\left\langle #1 \right\rangle}

\newcommand{\U}{\mathcal{U}}

\newcommand{\CC}{\mathcal{C}}

\newcommand{\Ese}[2]{\mathcal{S}\left(#1\mid #2\right)}

\newenvironment{proofof}{\par\noindent \textit{Proof of }}{\qed\par\bigskip}

\newcommand{\qand}{\quad \text{and} \quad}

\DeclareMathOperator{\Core}{Core}
\DeclareMathOperator{\Deg}{Deg}

\title[The isomorphism problem for rational group algebras of metacyclic nilpotent groups]{The isomorphism problem for rational group algebras of finite metacyclic nilpotent groups}

\author{Àngel García-Blázquez and Ángel del Río}
\thanks{Partially supported Grant PID2020-113206GB-I00 funded by MCIN/AEI/10.13039/501100011033 and by Fundación Séneca (22004/PI/22).}

\address{Departamento de Matem\'{a}ticas, Universidad de Murcia, 30100, Murcia, Spain} \email{angel.garcia11@um.es, adelrio@um.es}



\begin{document}

\maketitle


\begin{abstract}
We prove that if $G$ and $H$ are finite metacyclic groups with isomorphic rational group algebras and one of them is nilpotent then $G$ and $H$ are isomorphic.
\end{abstract}

\section{Introduction}

The following general problem has been largely studied since the seminal work of Graham Higman \cite{Higman1940Thesis,Higman1949Paper} and the influential paper of Richard Brauer \cite{Brauer1951}:

\begin{quote}
	\textbf{The Isomorphism Problem for Group Rings}: Given $R$ a commutative ring and $G$ and $H$ groups, does $RG$ and $RH$ being isomorphic as $R$-algebras implies that $G$ and $H$ are isomorphic as groups?
\end{quote}

Suppose that $G$ and $H$ are finite abelian groups. Higman proved that if $\Z G\cong \Z H$ then $G\cong H$.
This fails if $R=\C$ because if $G$ and $H$ are finite abelian group with the same order then $\C G\cong \C H$.
However, by a theorem of Perlis and Walker $\Q G\cong \Q H$ implies $G\cong H$ \cite{PerlisWalker1950}.
If now $G$ and $H$ are finite metabelian groups then still we have that $\Z G\cong \Z H$ implies $G\cong H$ \cite{Whitcomb}.
However, Dade showed two finite metabelian groups $G$ and $H$ such that $kG$ and $kH$ are isomorphic as algebras for every field $k$ \cite{Dade71}.

Observe that if $\Z G\cong \Z H$ then $RG\cong RH$ for every ring $R$. This explain why the positive results for the case where $R=\Z$ are more likely than for any other ring.
Likewise positive results are more likely in a prime field than in any other field with the same characteristic.
For a while it was expected that Isomorphisms Problem for Integral Group Ring may have a general positive answer at least for finite groups.
However Hertweck showed two non-isomorphic solvable groups $G$ and $H$ such that $\Z G\cong \Z H$ and hence $RG$ and $RH$ are isomorphic for every ring $R$ \cite{Hertweck2001}.

The aim of this paper is to contribute to the Isomorphism Problem for Group Rings with rational coefficients. The contrast between Perlis and Walker Theorem and the example of Dade suggests considering the class of metacyclic groups. The main result of the paper is the following, where $\pi_G$ denotes the set of primes $p$ for which $G$ has a normal Hall $p'$-subgroup:

\begin{maintheorem}\label{MainpiIgual}
Let $G$ and $H$ be a metacyclic finite groups such that $\Q G\cong \Q H$.
Then $\pi_G=\pi_H$ and the Hall $\pi_G$-subgroups of $G$ and $H$ are isomorphic.
\end{maintheorem}

As a direct consequence of \Cref{MainpiIgual} we obtain the following:

\begin{maincorollary}\label{Main}
If $G$ and $H$ are finite metacylic groups with $\Q G\cong \Q H$ and $G$ is nilpotent then $G\cong H$.
\end{maincorollary}

In \Cref{SectionNotation} we introduce the main notation of the paper and review some known results. Suppose that $G$ and $H$ are finite metacyclic groups such that $\Q G$ and $\Q H$ are isomorphic.
In \Cref{SectionpGroups} we prove that if $G$ and $H$ are $p$-groups then they are isomorphic.
In \Cref{SectionNilpotent} we prove \Cref{MainpiIgual}.

Observe that in \Cref{MainpiIgual} and \Cref{Main} it is not sufficient to assume that only one of the two groups $G$ or $H$ is metacyclic because the following groups
$$\GEN{a,b|a^{p^2}=b^p=1, a^b=a^{1+p}}, \quad \GEN{a,b|a^p=b^p=[b,a]^p=[a,[b,a]]=[b,[b,a]]=1}$$
have isomorphic rational group algebras while the first is metacyclic and the second is not.

\section{Notation and preliminaries}\label{SectionNotation}

\subsection{Number theory}

We adopt the convention that $0\not\in \N$ and prime means prime in $\N$.
Let $n\in \N$.
Then $\zeta_n$ denotes a complex primitive $n$-th root of unity and $\pi(n)$ denotes the set of prime divisors of $n$.
If $p$ is prime then $n_p$ denotes the greatest power of $p$ dividing $n$ and $v_p(n)=\log_p(n_p)$.
Moreover $v_p(0)=\infty$.
If $\pi$ is a set of primes then $n_{\pi}=\prod_{p\in \pi} n_p$.
If $m\in \Z$ with $\gcd(m,n)=1$, then $o_n(m)$ denotes the multiplicative order of $m$ modulo $n$, i.e. the smallest positive integer $k$ with $m^k\equiv 1 \mod n$.

If $A$ is a finite set then $|A|$ denotes the cardinality of $A$ and $\pi(A)=\pi(|A|)$.

If $x\in \Z\setminus \{0\}$ then we denote:
$$\Ese{x}{n}=\sum_{i=0}^{n-1} x^i =
\begin{cases} n, & \text{ if } x=1; \\
\frac{x^n-1}{x-1}, & \text{otherwise}.
\end{cases}$$

The notation $\Ese{x}{n}$ occurs in the following statement:
\begin{equation}\label{Potencia}
\text{If } g^h=g^x \text{ with } g \text{ and } h
\text{ elements of a group then } (hg)^n = h^n g^{\Ese{x}{n}}.
\end{equation}
The following lemma collects some properties of the operator $\Ese{-}{-}$.

\begin{lemma}\label{PropEse}
	Let $p$ be a prime, $R\in \Z$, $m\in \N$ and $a=v_p(R-1)\ge 1$.
	Then
	\begin{enumerate}
		\item \label{vpRm-1}
		$v_p(R^m-1)=\begin{cases}
		v_p(R-1)+v_p(m), & \text{if } p\ne 2 \text{ or } a\ge 2; \\
		v_p(R+1)+v_p(m), & \text{if } p=2, a=1 \text{ and } 2\mid m; \\
		1, & \text{otherwise}.
		\end{cases}$
		\item \label{op}
		$o_{p^m}(R)=\begin{cases}
		p^{\max(0,m-v_p(R-1))}, & \text{if } p\ne 2 \text{ or } a\ge 2; \\
		1, & \text{if } p=2, a=1 \text{ and } m\le 1; \\
		2^{\max(1,m-v_2(R+1))}, & \text{otherwise}.
		\end{cases}$
		\item Suppose that $a\le m$ and if $p=2$ then $a\ge 2$. Then the following hold:
		\begin{enumerate}
			\item\label{Potencias} $\{R^x + p^m\Z: x\ge 0\}=\{1+yp^a+p^m\Z : 0\le y < p^{m-a}\}$.
			\item\label{SumaPotencias}  If $n \in \N$ and $n\equiv kp^{m-a}\mod p^m$ then
			$$\Ese{R}{n} \equiv
			\begin{cases}  n+k2^{m-1} \mod 2^m, & \text{if } p=2 \text{ and } m>a;  \\
			n \mod p^m, & \text{otherwise}.\end{cases}$$
		\end{enumerate}
	\end{enumerate}
\end{lemma}

\begin{proof}
See \cite[Lemma~2.1]{GarciadelRio2022} and \cite[Lemma~8.2]{BrocheGLucasdelRio2022}.
\end{proof}

We will need the following formula:
\begin{equation}\label{Sumad2d}
\begin{split}
\sum_{d=0}^n d2^d &= \sum_{d=0}^n \sum_{i=0}^{d-1} 2^d
= \sum_{i=0}^{n-1} 2^{i+1} \sum_{d=i+1}^{n} 2^{d-i-1}
= \sum_{i=0}^{n-1} 2^{i+1} \sum_{j=0}^{n-i-1} 2^j
= \sum_{i=0}^{n-1} 2^{i+1} (2^{n-i}-1) \\
&= n2^{n+1}-2\sum_{i=0}^{n-1}2^i
= n2^{n+1}-2(2^n-1)=(n-1)2^{n+1}+2
\end{split}
\end{equation}


Recall that if $R, n\in\N$  with $\gcd(R,n)=1$ and $i\in \Z$ then the $R$-cyclotomic class modulo $n$ containing $i$ is the subset of $\Z$ formed by the integers $j$ such that $j\equiv iR^k \mod n$ for some $k\ge 0$.
The $R$-cyclotomic classes module $n$ form a partition of $\Z$ and each $R$-cyclotomic class modulo $n$ is a union of cosets modulo $n$.
More precisely, if $i$ and $j$ belong to the same $R$-cyclotomic class then
$\gcd(n,i)=\gcd(n,j)$ and if $d=\frac{n}{\gcd(n,i)}$ then the $R$-cyclotomic class module $n$ containing $i$ is the disjoint union of $i+n\Z,iR+n\Z,\dots,iR^{o_d(R)-1}+n\Z$.
Therefore the number of $R$-cyclotomic classes module $n$ is
\begin{equation}\label{Crn}
C_{R,n}=\sum_{d\mid n} \frac{\varphi(d)}{o_d(R)}.
\end{equation}
We will need a precise expression of this number for the case where $n$ is a power of $p$ and $R\equiv 1 \mod p$.

\begin{lemma}\label{CyclotomicClasses}
	Let $p$ be a prime and $R,m\in \N$ with $R\equiv 1 \mod p$. Then the number of $R$-cyclotomic classes modulo $p^m$ is
	$$C_{R,p^m} = \begin{cases} p^m, & \text{if } m\le v_p(R-1); \\
	1+2^{m-1}, & \text{if } p=2 \text{ and } 2\le m < v_2(R+1); \\
	1+2^{v_2(R+1)-1}(1+m-v_2(R+1)), & \text{if } p=2 \text{ and } 2\le v_2(R+1)\le m; \\
	p^{v_p(R-1)-1} (p+ (p-1) (m-v_p(R-1))), & \text{otherwise}.
	\end{cases}$$
\end{lemma}

\begin{proof}
	If $m\le v_p(R-1)$ then $o_d(R)=1$ for every divisor $d$ of $m$ and hence every $R$-cyclotomic class module $p^m$ is formed by one coset modulo $p^m$. Therefore, in that case $C_{R,p^m}=p^m$. Suppose otherwise that $m>v_p(R-1)$.

	Suppose that either $p$ is odd or $p=2$ and $R\equiv 1 \mod 4$.
	Using \Cref{PropEse}.\eqref{op} and \eqref{Crn} we have
	\begin{eqnarray*}
		C_{R,p^m}&=&\sum_{k=0}^m \frac{\varphi(p^k)}{p^{\max(0,k-v_p(R-1))}} =
		1+(p-1)\left(\sum_{k=1}^{v_p(R-1)} p^{k-1} + \sum_{k=v_p(R-1)+1}^m p^{v_p(R-1)-1} \right) \\
		&=& p^{v_p(R-1)} + (p-1) (m-v_p(R-1)) p^{v_p(R-1)-1}
		= p^{v_p(R-1)-1} (p+ (p-1) (m-v_p(R-1)))
	\end{eqnarray*}

	Otherwise, $p=2$ and $R\equiv -1\mod 4$. Then $2\le v_2(R+1)$ and $1=v_2(R-1)<m$.
	Using now \Cref{PropEse}.\eqref{op} and \eqref{Crn} we have
	$C_{R,2^m}= 2+ \sum_{k=2}^m \frac{\varphi(2^k)}{2^{\max(1,k-v_2(R+1))}}$
	Thus, if $m< v_2(R+1)$ the $C_{R,2^m}=2+\sum_{k=2}^m 2^{k-2} = 1+2^{m-1}$.
	Otherwise, i.e. if $m\ge v_2(R+1)$ then
	\begin{eqnarray*}
		C_{R,2^m} &=& 2+\sum_{k=2}^{v_2(R+1)} 2^{k-2} + \sum_{k=v_2(R+1)+1}^m 2^{v_2(R+1)-1}
		= 1+ 2^{v_2(R+1)-1}+(m-v_2(R+1))2^{v_2(R+1)-1} \\ &=& 1+2^{v_2(R+1)-1}(1+m-v_2(R+1)).
	\end{eqnarray*}
\end{proof}


\subsection{Group theory}

By default all the groups in this paper are finite.
We use standard notation for a group $G$ and $g,h\in G$:
$Z(G)=$ center of $G$, $G'=$ commutator subgroup of $G$, $\Aut(G)=$ group of automorphisms of $G$, $|g|=$ order of $g$, $g^h=g^{-1}hg$, $[g,h]=g^{-1}g^h$.
The notation $H\le G$ and $N\unlhd G$ means that $H$ is a subgroup of $G$ and $N$ is a normal subgroup of $G$.
If $H$ is a subgroup then $[G:H]$ denotes index of $H$ in $G$,
$N_G(H)$ the normalizer of $H$ in $G$ and $\Core_G(H)$ the core of $H$ in $G$, i.e. the greatest subgroup of $H$ that is normal in $G$.

If $\pi$ is a set of primes then $g_{\pi}$ and $g_{\pi'}$ denote the $\pi$-part and $\pi'$-part of $g$, respectively.
When $p$ is a prime we rather write $g_p$ and $g_{p'}$ than $g_{\{p\}}$ and $g_{\{p\}'\}}$.

Let $m$ be a positive integer. We let $C_m$ denote a generic cyclic group of order $m$.
For an integer $x$ coprime with $m$, let $\alpha_x$ denote the automorphism of $C_m$ given by $\alpha_x(a)=a^x$ for every $a\in A$ and $\sigma_x$ the automorphism of $\Q(\zeta_m)$ given by $\sigma_x(\zeta_m)=\zeta_m^x$.
Then $x\mapsto \alpha_x$ and $x\mapsto \sigma_x$ define isomorphisms
$\alpha:\Z_m^* \rightarrow \Aut(C_m)$ and $\sigma:\Z_m^*\rightarrow \Gal(\Q(\zeta_m)/\Q)$, where $\Z_m^*$ denotes the group of units of $\Z/m\Z$.
We are abusing the notation by treating elements of $\Z_m^*$ as integers.
In particular, $\alpha_x\mapsto \sigma_x$ defines an isomorphism $\Aut(C_m)\rightarrow \Gal(\Q(\zeta_m)/\Q)$. We abuse the notations by considering the latter as an identification so that if $\Gamma$ is a subgroup of $\Aut(C_m)$ then
$$\Q(\zeta_m)^{\Gamma}=\{a \in \Q(\zeta_m) : \sigma_x(a)=a \text{ for all } \alpha_x \in \Gamma\},$$
the Galois correspondent of $\Gamma$ considered as a subgroup of $\Gal(\Q(\zeta_m)/\Q)$.

\subsection{The finite metacyclic $p$-groups}

The finite metacyclic groups were classified by Hempel \cite{Hempel2000}.
Previously the finite metacyclic $p$-groups were classified by several means \cite{Zassenhaus1958,Lindenberg1971,Hall1959,Beyl1972,King1973,Liedahl1994,Liedahl1996,NewmanXu1988,Redei1989,Sim1994}. For our purpose we need the description of the finite metacyclic groups in terms of group invariants given in \cite{GarciadelRio2022} for the special case of $p$-groups. More precisely when \cite[Corollary~4.1]{GarciadelRio2022} is specialized to finite metacyclic $p$-groups one obtains the following:

\begin{theorem}\label{ClasiMeta}
Let $p$ be a prime integer. Then every finite metacyclic $p$-group is isomorphic to a group given by the following presentation
$$\mathcal{P}_{p,\mu,\nu,\sigma,\rho,\epsilon}=\GEN{a,b \mid a^{p^\mu}=1, b^{p^{\nu}}=a^{p^\sigma}, b^a=a^{\epsilon+p^\rho}}.$$
for unique non-negative integers $\mu,\nu,\sigma$ and $\rho$ and a unique $\epsilon \in \{1,-1\}$ satisfying the following conditions:
\begin{enumerate}[label=(\Alph*)]

    \item\label{rhomu} $\rho\le \mu$, if $\mu\ge 1$ then $\rho\ge 1$ and if $p=2\ge \mu$ then $\rho\ge 2$.
	\item\label{epsilon1} If $\epsilon=1$ then $\rho\le \sigma\le \mu \le \rho+\sigma$ and $\sigma\le \nu$.
	\item\label{epsilon-1} If $\epsilon=-1$ then
	\begin{enumerate}
		\item $p=2\le \rho\le \mu$, $\nu\ge 1$, $\mu-1\le \sigma\le \mu\le \rho+\nu\ne \sigma$ and
		\item if $2\ge \nu$ and $3\ge \mu$ then $\rho\le \sigma$,
	\end{enumerate}
\end{enumerate}
\end{theorem}

\subsection{Wedderburn decomposition of rational group algebras}

If $A$ is a finite dimensional central simple $F$-algebra for $F$ a field then $\Deg(A)$ denotes the degree of $A$, i.e. $\dim_F A=\Deg(A)^2$ (cf. \cite{Pierce1982}).

Let $F/K$ be a finite Galois field extension and let $G=\Gal(F/K)$.
Let $\U(F)$ denotes the multiplicative group of $F$.
If $f:G\times G\rightarrow \U(F)$ is a $2$-cocycle then $(F/K,f)$ denotes the crossed product
	$$(F/K,f)=\sum_{\alpha\in G} t_\alpha F, \quad xt_\alpha=t_\alpha \alpha(x), \quad  t_{\alpha}t_{\beta}=t_{\alpha\beta}f(\alpha,\beta), \quad (x\in F, \alpha,\beta\in G).$$
It is well known that $(F/K,f)$ is a central simple $K$-algebra and if $g$ is another $2$-cocycle then $(F/K,f)$ and $(F/K,g)$ are isomorphic as $K$-algebras if and only if $gf^{-1}$ is a $2$-coboundary. Therefore if $\overline f\in H^2(G,F)$ is represented by the coboundary $f$ then we denote $(F/K,\overline f)=(F/K,f)$.

If $G$ is cyclic of order $n$, generated by $\alpha$ and $a\in \U(K)$ then there is a cocycle $f:G\times G\rightarrow \U(K)$ given by
	$$f(\alpha^i,\alpha^j)=\begin{cases} 1, & \text{if } 0\le i,j,i+j<n; \\
	a, & \text{if } 0\le i,j<n\le i+j \end{cases}$$
Then the crossed product algebra $(F/K,f)$, is said to be a cyclic algebra, it is usually denoted $(F/K,\alpha,a)$ and it can be described as follows:
	$$(F/K,\alpha,a)=\sum_{i=0}^{n-1} u^i F = F[u\mid xu=u\alpha(x), u^n=a]$$

If $A$ is a semisimple ring then $A$ is a direct sum of central simple algebras.
This expression is called the \emph{Wedderburn decomposition} of $A$ and its simple factors  are called the \emph{Wedderburn components} of $A$.
The Wedderburn components of $A$ are the direct summands of the form $Ae$ with $e$ a primitive central idempotent of $A$.

Let $G$ be a finite group. Then $\Q G$ is semisimple and the center of each component $A$ of $\Q G$ is isomorphic to the field of character values $\Q(\chi)$ of any irreducible character $\chi$ of $G$ with $\chi(A)\ne 0$.
It is well known that $\Q(\chi)$ is a finite abelian extension of $\Q$ inside $\C$ and henceforth it is the unique subfield of $\C$ isomorphic to $\Q(\chi)$.
We will abuse the notation and consider $Z(A)$ as equal to $\Q(\chi)$.

An important tool for us is a technique to describe the Wedderburn decomposition of $\Q G$ introduced in \cite{OlivieridelRioSimon2004}. See also \cite[Section~3.5]{JespersdelRioGRG1}.
We recall here its main ingredients.

If $H$ is a subgroup of $G$ then $\widehat{H}$ denotes the element $|H|^{-1}\sum_{h\in H} h$ of the rational group algebra $\Q G$.
It is clear that $\widehat{H}$ is an idempotent of $\Q G$ and it is central in $\Q G$ if and only if $H$ is normal in $G$.

Let $N$ be a normal subgroup of $G$. Then the kernel of the natural homomorphism
$\Q G\rightarrow \Q(G/N)$ is $\Q G (1-\widehat{N})= \sum_{n\in N\setminus 1} \Q G(n-1)$.
Therefore $\Q G=\Q G\widehat{N} \oplus \Q G(1-\widehat{N})$ and
$\Q(G/N)\cong \Q G\widehat{N}$. Thus $\Q(G/N)$ is isomorphic to the direct sum of the Wedderburn components of $\Q G$ of the form $\Q Ge$ with $e$ a primitive central idempotent of $\Q G$ with $e\widehat{N}=e$.

We denote
	$$\varepsilon(G,N)=\begin{cases} \widehat{G}, & \text{if } G=N; \\
	\prod_{D/N\in M(G/N)} (\widehat{N}-\widehat{D}), & \text{otherwise}. \end{cases}$$
where $M(G/N)$ denote the set of minimal normal subgroups of $G$.
Clearly $\varepsilon(G,N)$ is a central idempotent of $\Q G$.

If $(H,K)$ is a pair of subgroups of $G$ with $K\unlhd H$ then we denote
	$$e(G,H,K) = \sum_{gC_G(\varepsilon(H,K))\in G/C_G(\varepsilon(H,K))} \varepsilon(H,K)^g.$$
Observe that $e(G,H,K)$ belongs to the center of $\Q G$. If moreover, $\varepsilon(H,K)^g\varepsilon(H,K)=0$ for every $g\in G\setminus C_G(\varepsilon(H,K))$, then $e(G,H,K)$ is an idempotent of $\Q G$.

A \emph{strong Shoda pair} of $G$ is a pair $(H,K)$ of subgroups of $G$ satisfying the following conditions:
\begin{itemize}
\item[(SS1)] $K\subseteq H \unlhd N_G(K)$,
\item[(SS2)] $H/K$ is cyclic and maximal abelian in $N_G(K)/K$,
\item[(SS3)] $\varepsilon(H,K)^g\varepsilon(H,K)$ for every $g\in G\setminus C_G(\varepsilon(H,K))$.
\end{itemize}

\begin{remark}\label{SSPRemark}
Suppose that $(H,K)$ is a strong Shoda pair of $G$ and let $m=[H:K]$ and $N=N_G(K)$.
Then $H/K\cong C_m$ and the action of $N$ by conjugation on $H$ induces a faithful action of $N/H$ on $\Q(\zeta_m)$. More precisely, if $n\in N$ then $h^nK=\alpha_r(hK)$ for some integer $r$, with $\gcd(r,m)=1$.
The map $nH\rightarrow \sigma_r$ defines an injective homomorphism $\alpha:N/H\rightarrow \Aut(\Q(\zeta_m))$. Let $F_{G,H,K}=\Q(\zeta_m)^{\Imagen \alpha}$. Then we have a short exact sequence \cite[Theorem~3.5.5]{JespersdelRioGRG1}:
	$$1\rightarrow H/K \cong \GEN{\zeta_m} \rightarrow N/K \rightarrow N/H\cong \Gal(\Q(\zeta_m)/F_{G,H,K})\rightarrow 1$$
which induces an element $\overline{f}\in H^2(N/H,\Q(\zeta_m))$. More precisely from an election of a set of representatives $\{c_u : u\in N/H\}$ of $H$ cosets in $N$, we define $f(u,v)=\zeta_m^k$ if $c_uc_v=c_{uv}h^k$. This defines an element of $H^2(N/H,\Q(\zeta_m))$ because another election yields to another $2$-cocycle differing in a $2$-coboundary.
Associated to $\overline{f}$ one has the crossed product algebra
\begin{equation*}
\begin{split}
& A(G,H,K)=(\Q(\zeta_m)/F_{G,H,K},\overline{f}) = \oplus_{u\in N/H} t_u \Q(\zeta_m), \\
& x t_u=t_u\sigma_u(x), \quad t_ut_v=t_{uv}f(u,v), \quad (x\in \Q(\zeta_m), u,v\in N/K).
\end{split}
\end{equation*}
\end{remark}

\begin{proposition}\cite[Proposition~3.4]{OlivieridelRioSimon2004} \cite[Theorem~3.5.5]{JespersdelRioGRG1} \label{SSPAlg}
Let $(H,K)$ be a strong Shoda pair of $G$ and let $m=[H:K]$, $n=[G:N_G(K)]$ and $e=e(G,H,K)$. Then $e$ is a primitive central idempotent of $\Q G$ and  $\Q Ge\cong M_n(A(G,H,K))$. Moreover,
$\Deg(\Q Ge)=[G:H]$, $Z(\Q Ge(G,H,K))\cong F_{G,H,K}$ and $\{g\in G : ge=e\}=\Core_G(K)$.
\end{proposition}

In the particular case where $G$ is metabelian all the Wedderburn components of $\Q G$ are of the form $A(G,H,K)$ for some special kind of strong Shoda pairs of $G$. More precisely we have the following (see \cite[Theorem~4.7]{OlivieridelRioSimon2004} or \cite[Theorem~3.5.12]{JespersdelRioGRG1}):

\begin{theorem}\label{SSPMetabelian}
Let $G$ be a finite group and let $A$ be a maximal abelian subgroup of $G$ containing $G'$.
Then every Wedderburn component of $\Q G$ is of the form $\Q Ge(G,H,K)$ for subgroups $H$ and $K$ satisfying the following conditions:
\begin{enumerate}
	\item $H$ is a maximal element in the set $\{B\le G : A\le B \text{ and } B'\le K\le B\}$.
	\item $H/K$ is cyclic.
\end{enumerate}
Moreover every pair $(H,K)$ satisfying (1) and (2) is a strong Shoda pair of $G$ and hence $\Q Ge(G,H,K)\cong M_n(A(G,H,K))$ with $n=[G:N_G(K)]$.
\end{theorem}

Suppose that $G$ is a finite metacyclic group and let $A$ be a cyclic normal subgroup of $G$ with $G/A$ cyclic. Then every Wedderburn component of $\Q G$ is of the form $\Q Ge(G,H,K)$ for $(H,K)$ be subgroups of $G$ satisfying the conditions of \Cref{SSPMetabelian}. Then, $N_G(K)/H$ is cyclic, say generated by $uH$ and of order $k$.
Moreover, $H/K$ is cyclic, say generated by $aK$, and normal in $N_G(K)/K$ so that $(aK)^{uK}=a^xK$ and $(uK)^k=a^y$ for some integers $x$ and $y$.
By \Cref{SSPAlg} we have
\begin{equation}\label{WCMetacyclic}
A(G,H,K)\cong (\Q(\zeta_m)/F,\sigma_x,\zeta_m^y) =
\Q(\zeta_m)[\overline{u} \mid \zeta_m \overline{u} = \overline{u} \zeta_m^x, \overline{u}^k=\zeta_m^y].
\end{equation}

\subsection{Tools for the Isomorphism Problem for group rings}

In this subsection we recall two results relevant for the Isomorphism Problem for rational group algebras. The first one
%
%
is a well known result of Artin which tell us what is the number of Wedderburn components of a rational group algebra. See \cite[Corollary~39.5]{CurtisReiner1962} or \cite[Corollary~7.1.12]{JespersdelRioGRG1}

\begin{theorem}[Artin]\label{Artin}
	If $G$ is a finite group then the number of Wedderburn components of $\Q G$ is the number of conjugacy classes of cyclic subgroups of $G$.
\end{theorem}

The second one is a consequence of the Perlis-Walker Theorem.

\begin{theorem}\label{Abelianizado}
	If $G$ and $H$ are finite groups with $\Q G\cong \Q H$ then $G/G'\cong H/H'$.
\end{theorem}

\begin{proof}
	Let $A(G)$ denote the kernel of the natural homomorphism $\Q G\rightarrow \Q(G/G')$.
	Then $A(G)$ is a the smallest ideal $I$ of $\Q G$ such that $(\Q G)/I$ is commutative. In particular, if $f:\Q G\rightarrow \Q H$ is an isomorphism then $f(A(G))=A(H)$ and therefore $f$ induces an isomorphism $\Q(G/G')\cong \Q(H/H')$. Then $G/G'\cong H/H'$ by the Perlis-Walker Theorem \cite{PerlisWalker1950}.
\end{proof}

\section{The Isomorphism Problem for finite metacyclic $p$-groups}\label{SectionpGroups}

In this section $p$ is a prime and we prove that the Isomorphism Problem for rational group algebras has positive solution for finite metacyclic $p$-groups.

All throughout $G$ is a finite metacyclic $p$-group. By \Cref{ClasiMeta}, $G\cong \mathcal{P}_{p,\mu,\nu,\sigma,\rho,\epsilon}$ for unique non-negative integers $\mu,\nu,\sigma$ and $\rho$ and unique $\epsilon\in \{1,-1\}$ satisfying conditions \ref{rhomu}-\ref{epsilon-1}.

The proof of the main result of this section relies in five technical lemmas.

\begin{lemma}\label{CyclicConjugate}
Suppose that $\epsilon=1$ and $\mu> 0$. Let $0\le d<\nu$ and for every $1\le i\le p^{\mu}$
	set
	$$l_i=\begin{cases}
	2^{\sigma}+i(2^{\nu-d}+2^{\mu-1}), & \text{if } p=2\nmid i \text{ and } \mu=\nu+\rho; \\
	p^{\sigma}+ip^{\nu-d}, & \text{otherwise},
	\end{cases}$$
	$$k_i=\min(\mu,v_p(l_i)) \qand
	h_i=\min(k_i,\rho+d,\rho+v_p(i)).$$
	Then $\GEN{b^{p^d}a^i}$ and $\GEN{b^{p^d}a^j}$ are conjugate in $G$ if and only if $i\equiv j \mod p^{h_i}$.
	In that case, $k_i=k_j$ and $h_i=h_j$.
\end{lemma}

\begin{proof}
	As $\mu>0$, by condition \ref{rhomu}, we also have $\rho>0$.
	Let $R=1+p^{\rho}$.
	By \Cref{PropEse}.\eqref{vpRm-1} we have $v_p(R^{p^d}-1)=d+\rho$ and by condition \ref{epsilon1} we have $\mu-(d+\rho)\le \nu-d$.
	Hence, applying \Cref{PropEse}.\eqref{SumaPotencias} with $a=d+\rho$ and $m=\nu+\rho>a$ we obtain the following for every $k\in \N$:
 $$\Ese{R^{p^d}}{kp^{\nu-d}} = \begin{cases}
 k2^{\nu+d}+k2^{\nu+\rho-1} \mod 2^{\nu+\rho}, & \text{if } p=2; \\
 kp^{\nu+\rho}, & \text{if} p\ne 2
 \end{cases}$$
 Then
	\begin{equation}\label{EseRpd}
	\Ese{R^{p^d}}{kp^{\nu-d}} \equiv \begin{cases}
	k2^{\nu-d}+k2^{\mu-1} \mod 2^{\mu}, & \text{if } p=2, \text{ and } \mu=\nu+\rho; \\
	kp^{\nu-d} \mod p^{\mu}, & \text{otherwise}.
	\end{cases}
	\end{equation}
	Moreover $a^{b^{p^d}}=a^{R^{p^d}}$ and hence, by \eqref{Potencia} we have
	\begin{equation}\label{bpdai^u}
	(b^{p^d}a^i)^{p^{\nu-d}} = b^{p^{\nu}}a^{i\Ese{R^{p^d}}{p^{\nu-d}}}=
	a^{p^{\sigma}+i\Ese{R^{p^d}}{p^{\nu-d}}}=a^{l_i}.
	\end{equation}

	Suppose that $\GEN{b^{p^d}a^i}$ and $\GEN{b^{p^d}a^j}$ are conjugate in $G$. Then there are integers $x,y,u$ with $p\nmid u$ such that $b^{p^d}a^j = ((b^{p^d}a^i)^u)^{b^ya^x}$. In particular $b^{p^d}\GEN{a}=b^{up^d}\GEN{a}$ and therefore $u\equiv 1 \mod p^{\nu-d}$.
	Write $u=1+vp^{\nu-d}$.
	Then
	$$(b^{p^d}a^i)^u = b^{p^d}a^i (b^{p^d}a^i)^{vp^{\nu-d}} = b^{p^d}a^{i+vl_i}$$
	Hence
	$$b^{p^d}a^j = (b^{p^d}a^{i+vl_i})^{b^ya^x} = b^{p^d} a^{(i+vl_i)R^y + x (1-R^{p^d})}.$$
	On the other hand, $R^y = 1 + Y p^{\rho}$ for some integer $Y$.
	Then
	$$j\equiv i + iYp^{\rho} + vl_iR^y +  x (1- R^{p^d}) \equiv i \mod p^{h_i}$$
	because $h_i=\min(k_i,\rho+d,r+v_p(i))=\min(\mu,v_p(l_i),v_p(1-R^{p^d}),\rho+v_p(i))$.

	Conversely suppose that $j\equiv i \mod p^{h_i}$ and consider the four possibilities for $h_i$ separately.
	Of course if $h_i=\mu$ then $b^{p^d}a^i=b^{p^d}a^j$.
	Suppose that $h_i=\rho+d$. Then $h_i=v_p(1-R^{p^d})$, by \Cref{PropEse}.\eqref{vpRm-1}.
	Therefore there is an integer $x$ such that $j\equiv i+x(1-R^{p^d}) \mod p^{\mu}$ and hence $(b^{p^d}a^i)^{a^x} = b^{p^d}a^{i+x(1-R^{p^d})}=b^{p^d}a^j$.
	Assume that $h_i=k_i=v_p(l_i)$.
	Then $j\equiv i+vl_i \mod p^{\mu}$ for some $v\in \N$. Hence using \eqref{bpdai^u} we have $(b^{p^d}a^i)^{1+vp^{\nu-d}}=b^{p^d}a^{i+vl_i}=b^{p^d}a^j$.
	Finally, suppose that $h_i=\rho+v_p(i)$.
	Then there is an integer $z$ such that $j\equiv i+zip^{\rho} \mod p^{\mu}$.
	Moreover, by \Cref{PropEse}.\eqref{Potencias}, there is a non-negative integer $y$ such that $R^y\equiv 1+zp^{\rho}$.
	Then $(b^{p^d}a^i)^{b^y}=b^{p^d}a^{iR^y}=b^{p^d} a^{i(1+zp^{\rho})}=b^{p^d}a^j$.

	For the last part, suppose that $\GEN{b^{p^d}a^i}$ and $\GEN{b^{p^d}a^j}$ are conjugate in $G$. Then, from \eqref{bpdai^u} we have $p^{\nu-d+\mu-k_i}=|b^{p^d}a^i|=|b^{p^d}a^j|=p^{\nu-d+\mu-k_j}$, so that $k_i=k_j$.
	Suppose that $h_i\ne h_j$. Then necessarily $v_p(i)\ne v_p(j)$ and, as $j\equiv i \mod p^{h_i}$, we have  $h_i\le v_p(i)<\rho+v_p(i)$. Interchanging the roles of $i$ and $j$ we also obtain $h_j\le v_p(j)<\rho+v_p(j)$. So that $h_i=\min(k_i,\rho+d)=\min(k_j,\rho+d)=h_j$, a contradiction.
\end{proof}

\begin{lemma}\label{CCCOdd}
	If $\epsilon=1$ then the number of conjugacy classes of cyclic subgroups of $G$ is $N=A_{\sigma}+A$ where
	\begin{eqnarray*}
		A_{\sigma}&=&p^{\rho-1}\sigma\left(1+(p-1)\frac{1+2\nu-\sigma}{2}\right) -\frac{p^{\rho+\sigma-\mu}}{p-1} \text{ and}\\
		A&=&\frac{3p^{\rho-1}-2}{p-1} +p^{\rho-1}\frac{6-\rho+2\nu\rho-\rho^2 +p(\rho^2 + 2\nu-3\rho-2\nu\rho+2)}{2}
	\end{eqnarray*}
\end{lemma}

\begin{proof}
For every $0\le d\le \nu$  we let $\CC_d$ denote the set of cyclic subgroups $C$ of $G$ satisfying $[C\GEN{a}:\GEN{a}]=p^{\nu-d}$.
Clearly $\CC_d$ is closed by conjugation in $G$.
We let $N_d$ denote the number of conjugacy classes of cyclic subgroups of $G$ belonging to $\CC_d$.
Then the number of conjugacy classes of subgroups of $G$ is $\sum_{d=0}^{\nu} N_d$.
For every $1\le i \le p^\mu$ we will use the notation $l_i$, $k_i$ and $h_i$ introduced in \Cref{CyclicConjugate}.

As $G/\GEN{a}$ is cyclic of order $p^{\nu}$, every element of $\CC_d$ is formed by the groups of the form $\GEN{b^{p^d}a^i}$ with $1\le i \le p^{\mu}$. In particular $N_{\nu}=\mu+1$, the number of subgroups of $\GEN{a}$.

	From now on we assume that $0\le d<\nu$.

	\textbf{Claim 1}. If $v_p(i)\ge \min(\sigma,\rho+d)$ then $\GEN{b^{p^d}a^i}$
	is conjugate to $\GEN{b^{p^d}}$ in $G$.

	Indeed, suppose that $v_p(i)\ge \min(\sigma,\rho+d)$.
	By \Cref{CyclicConjugate} we have to prove that $i\equiv 0 \mod p^{h_i}$, i.e. $h_i\le v_p(i)$.
	First of all observe that $v_p(i)\ge \min(\sigma,\rho+d)\ge 1$, because $1\le \rho\le \sigma$. Hence $l_i=p^\sigma+ip^{\nu-d}$. If $v_p(ip^{\nu-d})>\sigma$ then $v_p(l_i)=s$ and hence $h_i=\min(\sigma,\rho+d,\rho+v_p(i))=\min(\sigma,\rho+d)\le v_p(i)$, as desired.
	Suppose otherwise that $v_p(ip^{\nu-d})\le \sigma$.
	Then $v_p(i)<v_p(ip^{\nu-d})\le \sigma$ and hence, by hypothesis $\rho+d\le v_p(i)$. Then, by condition \ref{epsilon1} of \Cref{ClasiMeta} we have $v_p(ip^{\nu-d})\ge \rho+\nu\ge \mu \ge \sigma\ge v_p(ip^{\nu-d})$. Therefore  $v_p(ip^{\nu-d})=\rho+\nu=\mu=\sigma$ and $v_p(i)=\rho+d<\sigma$. Then $h_i=\min(\sigma,\rho+d)=\rho+d\le v_p(i)$, again as desired.

	\textbf{Claim 2}.
	If $1\le i,j \le p^{\mu}$, $v_p(i)< \min(\sigma,\rho+d)$ and $\GEN{b^{p^d}a^i}$ and $\GEN{b^{p^d}a^j}$ are conjugate in $G$ then $v_p(i)=v_p(j)$.

	Indeed, by \Cref{CyclicConjugate} we have $h_i=h_j$, which we denote $h$, and $i\equiv j \mod p^h$.
	By means of contradiction suppose that $v_p(i)\ne v_p(j)$.
	Then $h\le \min(v_p(i),v_p(j)) \le v_p(i) <\min(\sigma,\rho+d)$ and therefore $\min(\mu,v_p(l_j),\rho+d)  = \min(\mu,v_p(l_i),\rho+d) = h <\min(\sigma,\rho+d)$.
	Thus $v_p(l_i)=v_p(l_j)=h\le \min(v_p(i),v_p(j))$. However $v_p(ip^{\nu-d})>v_p(i)$ and if $p=2\ne i$ then $v_2(i(2^{\nu-d}+2^{\mu-1}))>v_2(i)$.
	Therefore $v_p(l_i-p^{\sigma})>v_p(i)\ge h = v_p(l_i)$. Then $h=v_p(l_j)=v_p(l_i)=\sigma\ge \min(\sigma,\rho+d)$, a contradiction.

	\medskip
	We use Claims 1 and 2 and \Cref{CyclicConjugate} as follows: For every $0\le h < \min(\sigma,\rho+d)$ let
	$$X_h = \{ i \in \Z : 1\le i \le p^{\mu} \text{ and } v_p(i)=h\}$$
	and consider the equivalence relation in $X_h$ given by
	$$i\sim_d j \text{ if and only if } k_i=k_j (=k) \text{ and } i\equiv j \mod p^{\min(k,\rho+d,\rho+h)}.$$
	Let $N_{d,h}$ be the number of $\sim_d$-equivalence classes in $X_h$.
	By \Cref{CyclicConjugate} and Claim 2, if $i\in X_h$, $1\le j\le p^\mu$, $v_p(i)<\min(\sigma,\rho+d)$ and $\GEN{b^{p^d}a_i}$ and $\GEN{b^{p^d}j}$ are conjugate in $G$ then $j\in X_h$ and $i$ and $j$ belong to the same $\sim_d$-class. Therefore, using also Claim 1 we have
	\begin{equation}\label{Nd=SumaNdh}
	N_d=1+\sum_{h=0}^{\min(\sigma,\rho+d)-1} N_{d,h}.
	\end{equation}

	Our next goal is obtaining a formula for $N_{d,h}$ and for that we consider three cases:

	\noindent\textbf{Case 1}: Suppose that $d\le \nu-\rho$.

	Let $h\in X_h$. We claim that $k_i=\min(\sigma,\rho+d,\rho+h)$.
	This is clear if $v_p(l_i)=\sigma$.
	Suppose that $v_p(l_i)>\sigma$.
	Then $v_p(l_i-p^{\sigma})=\sigma$.
	If $h=0$ then, as $\rho\le \sigma$ we have $k_i=\rho = \min(\sigma,\rho+d,\rho+h)$ as desired. Otherwise $l_i-p^\sigma=ip^{\nu-d}$, so that $h+\nu-d=\sigma$ and, by assumption we have $\rho+h=\rho+d-\nu+\sigma\le \sigma$.
	Then again $k_i=\min(\sigma,\rho+d,\rho+h)$.
	Finally, suppose that $v_p(l_i)<\sigma$. Then $v_p(l_i-p^\sigma)=v_p(l_i)$.
	If $l_i-p^{\sigma}=ip^{\nu-d}$ then $h+\rho\le h+\nu-d=v_p(l_i)<\sigma\le \mu$ and hence $k_i=\min(\rho+d,\rho+h) = \min(\sigma,\rho+d,\rho+h)$.
	Otherwise $p=2$, $h=0$, $\mu=\nu+\rho$ and $l_i-2^{\sigma}=i(2^{\nu-d}+2^{\mu-1})$.
	Then $\mu\ge \sigma>v_2(l_i)=v_2(2^{\nu-d}+2^{\mu-1})\ge \nu-d\ge \rho=\rho+h$, because $\nu-d=\mu-\rho-d\le \mu-\rho\le \mu-1$. Then $ k_i=\rho=\min(\sigma,\rho+d,\rho+h)$.
	So all the cases $k_i =\min(\sigma,\rho+d,\rho+h)$, as desired.

	Combining \Cref{CyclicConjugate} with the claim in the previous paragraph we deduce that for $d\le \nu-\rho$ and $h<\min(\sigma,\rho+d)$, the $\sim_d$-equivalence classes of $X_h$ have $p^{\mu-\min(\sigma,\rho+d,\rho+h)}$ elements.
	Thus for each $d\le \nu-\rho$ and $0\le h < \min(\sigma,\rho+d)$ we have
	$$N_{d,h}=\frac{\varphi(p^{\mu-h})}{p^{\mu-\min(\sigma,\rho+d,\rho+h)}}=
	(p-1)p^{\min(\sigma,\rho+d,\rho+h)-h-1}$$
	As, by Claim 1, for a fixed $d\mid \mu$, all the cyclic groups $\GEN{b^{p^d}a^i}$ with $v_p(i)\ge \min(\sigma,\rho+d)$ are conjugate we have
	\begin{align}\label{dMenorOIgualn-r}
	\begin{split}
	& \sum_{d=0}^{\nu-\rho} N_d =
	\sum_{d=0}^{\nu-\rho} \left(1+\sum_{h=0}^{\min(\sigma,\rho+d)-1} N_{d,h}\right)
	=
	\sum_{d=0}^{\sigma-\rho} \left(1+\sum_{h=0}^{d-1} (p-1)p^{\rho-1} + \sum_{h=d}^{d+\rho-1} (p-1)p^{\rho+d-h-1}\right) \\
	 &+\sum_{d=\sigma-\rho+1}^{\nu-\rho} \left(1+\sum_{h=0}^{\sigma-\rho-1} (p-1)p^{\rho-1}+\sum_{h=\sigma-\rho}^{\sigma-1} (p-1)p^{\sigma-h-1}\right) = (\nu-\rho+1)+\\
	& \sum_{d=0}^{\sigma-\rho} \left( d(p-1)p^{\rho-1} + (p-1)\sum_{x=0}^{\rho-1} p^x \right)
	+ \sum_{d=\sigma-\rho+1}^{\nu-\rho}\left( (\sigma-\rho)(p-1)p^{\rho-1} + (p-1)\sum_{x=0}^{\rho-1} p^x \right) \\
	&= (\nu-\rho+1) + \frac{(\sigma-\rho)(\sigma-\rho+1)}{2}(p-1)p^{\rho-1} +
	 (\nu-\sigma)(\sigma-\rho)(p-1)p^{\rho-1} + (\nu-\rho+1)(p^{\rho}-1) \\
	&= p^{\rho-1} \left( (\sigma-\rho)(p-1) \frac{1+2\nu-\rho-\sigma}{2} + (\nu-\rho+1)p \right)
	\end{split}
	\end{align}

	\noindent\textbf{Case 2}: Suppose that $\nu-\rho<d\le \nu-1$ and $h\ne \sigma+d-\nu$.

	Let $i\in X_h$. Then $v_p(p^{\sigma}+ip^{\nu-d})=\min(\sigma,h+\nu-d)$.
	If $v_p(l_i)\ne \min(\sigma,h+\nu-d)$ then $p=2\ne i$, $\mu=\nu+\rho$ and $l_i=2^{\sigma}+i(2^{\nu-d}+2^{\mu-1})$.
	Then, from $\rho\ge 1$ and $d>\nu-\rho\ge 0$ we deduce $\mu-1 = v_2(l_i-(2^{\sigma}+i2^{\nu-d})) = \min(v_2(l_i),v_2(2^{\sigma}+i2^{\nu-d})) = \min(v_2(l_i),\sigma,\nu-d)\le \nu-d=\mu-\rho-d< \mu-1$, a contradiction.
	This proves that $v_p(l_i)=\min(\sigma,h+\nu-d)$.
	Therefore $h_i=\min(\mu,v_p(l_i),\rho+d,\rho+h)=\min(\sigma,h+\nu-d,\rho+d,\rho+h) = \min(\sigma,h+\nu-d,\rho+h)= \min(\sigma,h+\nu-d)\le \mu$ because $\sigma\le \nu < \rho+d$ and $h+\nu-d< h+\rho$, by condition \ref{epsilon1} in \Cref{ClasiMeta} and the assumption.
	Hence, by \Cref{CyclicConjugate}, each class inside $X_h$ with $\sigma>h\ne \sigma+d-\nu$ contains
	$p^{\mu-\min(\sigma,h+\nu-d)}$ elements.
	This proves the following
	$$\text{if } \sigma> h\ne \sigma+d-\nu \text{ then } N_{d,h}=\frac{\varphi(p^{\mu-h})}{p^{\mu-\min(\sigma,h+\nu-d)}}=
	(p-1)p^{\min(\sigma-h,\nu-d)-1}.$$
Then
	\begin{align}\label{SumaNoIgual}
	\begin{split}
	& \sum_{d=\nu-\rho+1}^{\nu-1} \left( 1 + \sum_{h=0,h\ne \sigma+d-\nu}^{\sigma -1} N_{d,h}\right)\\
	& = \sum_{d=\nu-\rho+1}^{\nu-1} \left(1+(p-1)\left(\sum_{h=0}^{\sigma+d-\nu-1} p^{\nu-d-1} + \sum_{h=\sigma+d-\nu+1}^{\sigma-1} p^{\sigma-h-1}\right)\right) \\
	& = \sum_{x=1}^{\rho-1} \left(1+(p-1)\sum_{h=0}^{\sigma-x-1} p^{x-1} + (p-1)\sum_{h=\sigma-x+1}^{\sigma-1} p^{\sigma-h-1}\right) \\
	& = \sum_{x=1}^{\rho-1} \left(1+(\sigma-x)(p-1)p^{x-1} + (p-1)\sum_{y=0}^{x-2} p^y\right) = \sum_{x=1}^{\rho-1} \left( (\sigma-x)p^x-(\sigma-x)p^{x-1} + p^{x-1}\right) \\
	& = \sum_{x=1}^{\rho-1} (\sigma-x)p^x - \sum_{x=0}^{\rho-2} (\sigma-x-1)p^x + \sum_{x=0}^{\rho-2} p^x
	= (\sigma-\rho+1)p^{\rho-1} + 2\sum_{x=1}^{\rho-2} p^x - (\sigma-1) +1\\
	& = (1-\rho)p^{\rho-1} + \sigma(p^{\rho-1}-1)+ 2\frac{p^{\rho-1}-1}{p-1}
	\end{split}
	\end{align}

	\noindent \textbf{Case 3}: Finally, suppose that $\nu-\rho<d\le \nu-1$ and $h=\sigma+d-\nu$.

	Then, $h<\sigma$ and by condition \ref{epsilon1} if $i\in X_h$ then $v_p(i)=h\ge \rho+d-\nu>0$ and hence
	$l_i= p^{\sigma} + ip^{\nu-d}=p^{\sigma}(1+ip^{-h})$.
	Therefore $v_p(l_i)=\sigma+v_p(1+ip^{-h})$.
	Also, by condition \ref{epsilon1} we have $\rho\le \sigma\le \nu$, and therefore $h\le d$.
	Thus $h_i=\min(k_i,\rho+h)$.
	Observe that, as $1\le i \le p^{\mu}$, we have that $0\le v_p(1+ip^{-h})\le \mu-h$.
	For $0\le l \le \mu-h$ we set
	$$Y_l=\{i\in X_h : v_p(1+ip^{-h})=l\} \qand Z_l=\bigcup_{t=l}^{\mu-h} Y_t.$$
	The sets $Y_l$ with $l=0,1,\dots, \mu-h$ form a partition of $X_h$.
	A straightforward argument show that
	$$|Y_l|=\begin{cases}
	(p-2)p^{\mu-h-1}, & \text{if } l=0; \\
	\varphi(p^{\mu-h-l}), & \text{if } 1\le l < \mu-h; \\
	1, & \text{if } l=\mu-h;
	\end{cases} \qand
	|Z_l| = \begin{cases} \varphi(p^{\mu-h}), & \text{if } l=0; \\ p^{\mu-h-l}, & \text{if } 1\le l \le \mu-h.\end{cases}$$
	For each $i\in Y_l$ we have $k_i=\min(\mu,\sigma+l)$. 	Therefore, if $i\in Y_l$ then
	$$h_i=\begin{cases} \min(\mu,\rho+h), & \text{if } i\in Z_{\mu-\sigma}; \\
	\min(\sigma+l,\rho+h), & \text{otherwise}.\end{cases}$$
	By \Cref{CyclicConjugate}, each $\sim_d$-class inside $X_h$ is contained either in some $Y_l$
	with $l<\mu-\sigma$ or in $Z_{\mu-\sigma}$. Moreover two elements $i$ and $j$ in $Y_l$ with $l<\mu-\sigma$ belong to the same class if and only if $i\equiv j \mod p^{\min(\sigma+l,\rho+h)}$ while two elements in $Z_{\mu-\sigma}$ are in the same class if and only if $i\equiv j \mod p^{\min(\mu,\rho+h)}$.
	Recalling that $h=\sigma+d-\nu$ we deduce that if $l< \min(\mu-\sigma,\rho+d-\nu)$ then each class inside $Y_l$ has cardinality $p^{\mu-(\sigma+l)}$, while every class contained in $Z_{\min(\mu-\sigma,\rho+d-\nu)}$ has cardinality $p^{\mu-\min(\mu,\rho+h)}$.
	Having in mind that $\frac{|Z_{\min(\mu-\sigma,\rho+d-\nu)}|}{p^{\mu-\min(\mu,\rho+\sigma+d-\nu)}}=\frac{|Z_{\min(\mu-\sigma,\rho+d-\nu)+1}|}{p^{\mu-\min(\mu,\rho+\sigma+d-\nu)}}+\frac{|Y_{\min(\mu-\sigma,\rho+d-\nu)}|}{p^{\mu-(\sigma+\min(\mu-\sigma,\rho+d-\nu))}}$ we have
	\begin{align*}
		& N_{d,\sigma+d-\nu}= \frac{|Z_{\min(\mu-\sigma,\rho+d-\nu)+1}|}{p^{\mu-\min(\mu,\rho+\sigma+d-\nu)}} + \sum_{l=0}^{\min(\mu-\sigma,\rho+d-\nu)}  \frac{|Y_l|}{p^{\mu-(\sigma+l)}} \\
		& =   p^{\nu-d-1} + (p-2)p^{\nu-d-1} + \sum_{l=1}^{\min(\mu-\sigma,\rho+d-\nu)} \frac{\varphi(p^{\mu-\sigma+\nu-d-l})}{p^{\mu-\sigma-l}} \\
		&= (p-1)p^{\nu-d-1} + \min(\mu-\sigma,\rho+d-\nu)(p-1)p^{\nu-d-1} = (1+\min(\mu-\sigma,\rho+d-\nu))(p-1)p^{\nu-d-1}
	\end{align*}
	Thus
	\begin{align}\label{SumaIgual}
	\begin{split}
	& \sum_{d=\nu-\rho+1}^{\nu-1} N_{d,\sigma+d-\nu} =
	(p-1)\sum_{d=\nu-\rho+1}^{\nu-1} (1+\min(\mu-\sigma,\rho+d-\nu))p^{\nu-d-1} \\
	&= (p-1)\sum_{x=0}^{\rho-2}(1+\min(\mu-\sigma,\rho-x-1))p^x \\
	&= (p-1)\left(\sum_{x=0}^{\rho+\sigma-\mu-2}(1+\mu-\sigma) p^x +  \sum_{x=\rho+\sigma-\mu-1}^{\rho-2}(\rho-x)p^x\right) \\
	&= (1+\mu-\sigma) (p^{\rho+\sigma-\mu-1}-1) + \sum_{x=\rho+\sigma-\mu-1}^{\rho-2}(\rho-x)p^{x+1}-\sum_{x=\rho+\sigma-\mu-1}^{\rho-2}(\rho-x)p^x\\
	&= (1+\mu-\sigma) (p^{\rho+\sigma-\mu-1}-1) + \sum_{x=\rho+\sigma-\mu}^{\rho-1}(\rho-x+1)p^x-\sum_{x=\rho+\sigma-\mu-1}^{\rho-2}(\rho-x)p^x\\
	&= (1+\mu-\sigma) (p^{\rho+\sigma-\mu-1}-1) + 2 p^{\rho-1} + \sum_{x=\rho+\sigma-\mu}^{\rho-2} p^x - (\mu+1-\sigma)p^{\rho+\sigma-\mu-1} \\
	&= (\sigma-1-\mu) + 2 p^{\rho-1} + p^{\rho+\sigma-\mu}\sum_{x=0}^{\mu-\sigma-2} p^x
	= (\sigma-1-\mu) + 2 p^{\rho-1} + p^{\rho+\sigma-\mu}\frac{p^{\mu-\sigma-1}-1}{p-1} \\
	&= (\sigma-1-\mu) + 2 p^{\rho-1} + \frac{p^{\rho-1}-p^{\rho+\sigma-\mu}}{p-1}.
	\end{split}
	\end{align}

Combining \eqref{Nd=SumaNdh}, \eqref{dMenorOIgualn-r}, \eqref{SumaNoIgual}. and \eqref{SumaIgual}, and recalling that $N_\nu=\mu+1$, we finally obtain that the number of conjugacy classes of cyclic subgroups of $G$
\begin{align}
\begin{split}
\sum_{d=0}^{\nu} N_d=& \mu+1+\sum_{d=0}^{\nu-\rho} N_d +
\sum_{d=\nu-\rho+1}^{\nu-1} \left( 1+\sum_{h=0,h\ne \sigma+d-\nu}^{\min(\sigma,\rho+d)} N_{d,h}  + N_{d,\sigma+d-\nu}\right) \\
=& \mu+1+p^{\rho-1} \left( (\sigma-\rho)(p-1) \frac{1+2\nu-\rho-\sigma}{2} + (\nu-\rho+1)p \right)\\
&+(1-\rho)p^{\rho-1} + \sigma(p^{\rho-1}-1)+ 2\frac{p^{\rho-1}-1}{p-1} +
(\sigma-1-\mu) + 2 p^{\rho-1} + \frac{p^{\rho-1}-p^{\rho+\sigma-\mu}}{p-1}\\
=& p^{\rho-1}\sigma\left[1+(p-1)\frac{1+2\nu-\rho-\sigma}{2}\right] +
p^{\rho-1}\left( -\rho(p-1) \frac{1+2\nu-\rho-\sigma}{2} + (\nu-\rho+1)p \right)\\
&+(1-\rho)p^{\rho-1} + 2\frac{p^{\rho-1}-1}{p-1} + 2 p^{\rho-1} + \frac{p^{\rho-1}-p^{\rho+\sigma-\mu}}{p-1}\\
=& p^{\rho-1}\sigma\left[1+(p-1)\frac{1+2\nu-\sigma}{2}\right]  -\frac{p^{\rho+\sigma-\mu}}{p-1}\\
&+ \frac{3p^{\rho-1}-2}{p-1} +
p^{\rho-1}\frac{6-2\rho -\rho(p-1) (1+2\nu-\rho) + 2(\nu-\rho+1)p}{2}\\
=& p^{\rho-1}\sigma\left[1+(p-1)\frac{1+2\nu-\sigma}{2}\right]  -\frac{p^{\rho+\sigma-\mu}}{p-1}\\
&+ \frac{3p^{\rho-1}-2}{p-1} +
p^{\rho-1}\frac{6-\rho+2\nu\rho-\rho^2 +p(\rho^2 + 2\nu-3\rho-2\nu\rho+2)}{2} = A_\sigma+A.
\end{split}
\end{align}
\end{proof}

\begin{lemma}\label{CCN}
	If $\epsilon=-1$ then the number of conjugacy classes of $G$ is $3\cdot 2^{\nu-1} + 2^{\rho-1}(3\cdot 2^{\nu-1} - 2^{\nu+\rho-\mu})$.
\end{lemma}
\begin{proof}
Every element of $G$ is of the form $b^ja^i$ with $0\le j <2^{\nu}$ and $0\le i < 2^{\mu}$. As $G/\GEN{a}$ is cyclic of order $2^{\nu}$, if $b^ja^i$ and $b^{j'}a^{i'}$ are conjugate then $j=j'$.
Let $R=-1+2^\rho$ and $d_j=\gcd(2^{\mu},R^j-1)$.
Then
	$$(b^ja^i)^{b^ya^x} = b^j a^{x(1-R^j)+iR^{y}}$$
This shows that $b^ja^i$ and $b^ja^{i'}$ belong to the same conjugacy class if and only if the congruence equation $X(1-R^j)+iR^y\equiv i' \mod 2^{\mu}$ has a solution if and only if $i'\equiv iR^y \mod d_j$ if and only if $i$ and $i'$ belong to the same $R$-cyclotomic class modulo $d_j$.
Therefore the number of conjugacy classes of elements of the form $b^ja^i$ is $C_{R,d_j}$, the number of $R$-conjugacy classes modulo $d_j$.
By \Cref{PropEse}.\eqref{vpRm-1} we have
	$$d_j=\begin{cases} 2, & \text{if } 2\nmid j; \\ 2^{\min(\mu,\rho+v_2(j))}, & \text{otherwise}.\end{cases}$$
Using \Cref{CyclotomicClasses} and having in mind that $R+1=2^{\rho}$ with $2\le \rho\le \mu$, we have
	$$C_{R,d_j} = \begin{cases} 2, & \text{if } j \nmid 2; \\
	1+2^{\rho-1}(1+\min(\mu-\rho,v_2(j))), & \text{otherwise}.\end{cases}$$
Therefore, as the number of integers $1\le j \le 2^{\nu}$ with $v_2(j)=k$ is $\varphi(2^{n-k})$, we deduce that that the number of conjugacy classes of $G$ is
	\begin{eqnarray*}
		\sum\limits_{j=1}^{2^{\nu}} C_{R,d_j} &=& 2^{\nu} + \sum\limits_{k=1}^{\nu}\varphi(2^{\nu-k})(1+2^{\rho-1}(1+\min(\mu-\rho,k))) \\ &=& 2^{\nu} + (1+2^{\rho-1})\sum\limits_{k=1}^{\nu}\varphi(2^{\nu-k}) + 2^{\rho-1}\sum\limits_{k=1}^{\nu}\varphi(2^{\nu-k}) \min (\mu-\rho,k)\\
		&=& 2^{\nu} + (1+2^{\rho-1})2^{\nu-1} + 2^{\rho-1}\sum\limits_{k=1}^{\mu-\rho-1}2^{\nu-k-1}k + 2^{\rho-1}\sum\limits_{k=\mu-\rho}^{\nu}\varphi(2^{\nu-k})(\mu-\rho)    \\
		&=& 3\cdot 2^{\nu-1} + 2^{\rho+\nu-2} + 2^{\rho-1}\sum\limits_{k=1}^{\mu-\rho-1}2^{\nu-k-1}k + (\mu-\rho)2^{\nu+2\rho-\mu-1}
	\end{eqnarray*}
	We calculate separately the sum in the third summand.
	If $\mu=\rho$ or $\mu=\rho+1$ the summand is zero, so we assume $\mu \ge \rho+2$. Then, using \eqref{Sumad2d} we obtain
	\begin{eqnarray*}
		2^{\rho-1}\sum\limits_{k=1}^{\mu-\rho-1}2^{\nu-k-1}k  &=&  2^{\nu+2\rho-\mu-1}\sum\limits_{k=1}^{\mu-\rho-1}2^{\mu-\rho-k-1}k\\
		&=& 2^{\nu+2\rho-\mu-1}\sum\limits_{i=0}^{\mu-\rho-2}2^{i}(\mu-\rho-i-1)\\
		&=& 2^{\nu+2\rho-\mu-1}\left( (\mu-\rho-1)\sum\limits_{i=0}^{\mu-\rho-2}2^{i}-\sum\limits_{i=0}^{\mu-\rho-2}2^{i}i \right)\\
		&=& 2^{\nu+2\rho-\mu-1}\left( (\mu-\rho-1)(2^{\mu-\rho-1}-1) -((\mu-\rho-3)2^{\mu-\rho-1}+2) \right)\\
		&=& 2^{\nu+2\rho-\mu-1}(2^{\mu-\rho} -\mu+\rho-1).
	\end{eqnarray*}
	Observe that replacing $\mu$ by $\rho$ or $\rho+1$ in the previous expression the result is zero
	So there is not need to distinguish cases and we finally obtain the desired formula for the number of conjugacy classes of $G$:
	\begin{eqnarray*}
		\sum\limits_{j=1}^{2^{\nu}} C_{R,d_j} &=& 3\cdot 2^{\nu-1} + 2^{\rho+\nu-2} + (\mu-\rho)2^{\nu+2\rho-\mu-1} +  2^{\nu+2\rho-\mu-1}(2^{\mu-\rho} -\mu+\rho-1)\\
		&=& 3\cdot 2^{\nu-1} + 3\cdot 2^{\rho+\nu-2} - 2^{\nu+2\rho-\mu-1}\\
		&=& 3\cdot 2^{\nu-1} + 2^{\rho-1}(3\cdot 2^{\nu-1} - 2^{\nu+\rho-\mu})
	\end{eqnarray*}
\end{proof}

\begin{lemma}\label{SpecialcaseNr}
Suppose that $\epsilon=-1$, $\rho\ge \mu-1$  and $\mu\ge 3$.
Then the following statements hold:
	\begin{enumerate}
		\item $\Q G$ has a simple component with center $\Q(\zeta_{2^{\mu}}+\zeta_{2^{\mu}}^{-1})$ if and only if $\rho=\sigma=\mu$.
		\item $\Q G$ has a simple component with center $\Q(\zeta_{2^{\mu}}-\zeta_{2^{\mu}}^{-1})$ if and only if $\rho=\mu-1$ and $\sigma=\mu$.
	\end{enumerate}
\end{lemma}

\begin{proof}
	Let $H=C_G(a)$ and $K_0=\GEN{b^2}$.
	The assumption $\rho\ge \mu-1$ implies that $H=\GEN{a,b^2}$ is a maximal abelian subgroup of $G$.
	Then $(H,K_0)$ satisfy the conditions in \Cref{SSPMetabelian} and hence $\Q Ge(G,H,K_0)$ is a simple component of $\Q G$. Moreover, by \Cref{SSPAlg} we have
	$$Z(\Q G e(G,H,K_0)) \cong \begin{cases} \Q(\zeta_{2^{\mu}}+\zeta_{2^{\mu}}^{-1}), & \text{if } \rho=\sigma=\mu; \\
	\Q(\zeta_{2^{\mu}}-\zeta_{2^{\mu}}^{-1}), & \text{if } \rho=\mu-1 \text{ and } \sigma=\mu; \\
	\Q(\zeta_{2^{\mu-1}}+\zeta_{2^{\mu-1}}^{-1}), & \text{if } \rho=\sigma=\mu-1.
	\end{cases}$$
	This proves the reverse implication of (1) and (2).

	Conversely suppose that $A$ is a simple component of $\Q G$ with center $\Q(\zeta_{2^{\mu}}+\zeta_{2^{\mu}}^{-1})$ or $\Q(\zeta_{2^{\mu}}-\zeta_{2^{\mu}}^{-1})$.
	Since $\mu\ge 3$, this fields are not cyclotomic extensions of $\Q$ and therefore $A$ is not commutative, for otherwise $A$ will be a Wedderburn component of $\Q(G/G')$ and the Wedderburn components of a commutative rational group algebra are cyclotomic extensions of $\Q$.
	As $H$ is maximal abelian in $G$ and $G/H\cong C_2$ there is a pair $(H_1,K)$ of subgroups of $G$ satisfying the conditions of \Cref{SSPMetabelian} and $H_1 \in \{H,G\}$. However, $H_1\ne G$ because $A$ is not commutative.
	Therefore $H=H_1$.
	If $K$ is not normal in $G$ then $N_G(K)=H$ and hence $A\cong M_2(\Q(\zeta_{[H:K]}))$ contradicting the fact that the center of $A$ is not cyclotomic.
	Thus $K$ is normal in $G$ and the center of $A$ has index $2$ in $\Q(\zeta_{[H:K]})$.
	By \Cref{SSPAlg}, $\varphi([H:K])=2\dim Z(A)=2^{\mu-1}$ and hence $[H:K]=2^{\mu}$.
	Another consequence of \Cref{SSPAlg} and the fact that $A$ is not commutative is that $H\ne \GEN{K,b^2}$ and as $H/K=\GEN{aK,b^2K}$ is a cyclic $2$-group it follows that $H=\GEN{K,a}$.
	As $[H:K]=2^{\mu}=|a|$ we have $a^{2^{\mu-1}}\not\in K$.
	Thus $G'\cap K=1$. As $K$ is normal in $G$, it follows that $K\subseteq Z(G)=\GEN{a^{2^{\mu-1}},b^2}$. If $\sigma=\mu-1$ then $Z(G)=\GEN{b^2}$ and its order is $2^{\nu}$. Then $K=\GEN{b^4}$ which is not possible because $H/\GEN{b^4}$ is not cyclic.
	Thus $\sigma=\mu$ and $Z(G)=\GEN{a^{2^{\mu-1}}}\times \GEN{b^2}$. Then $K=\GEN{b^2}$ or $K=\GEN{a^{2^{\mu-1}}b^2}$.
	Arguing as in the first paragraph we deduce that $Z(\Q Ge(G,H,K))=\Q(\zeta_{2^{\mu}}+\zeta_{2^{\mu}}^{-1})$ if $\rho=\mu$ and $Z(\Q Ge(G,H,K))=\Q(\zeta_{2^{\mu}}-\zeta_{2^{\mu}}^{-1})$ if $\rho=\mu-1$.
\end{proof}

\begin{lemma}\label{GroupsNscalculus}
Suppose that $\epsilon=-1$ and $\rho<\mu<\nu+\rho$.
Let $F=\{\alpha \in \Q(\zeta_{2^{\mu}}) : \sigma_{-1+2^\rho}(\alpha)=\alpha\}$.
Then $\Q G$ has a simple component of degree $2^{\mu-\rho}$ and center $F$ if and only if $\sigma=\mu$.
\end{lemma}

\begin{proof}
Let $H=\GEN{a,b^{2^{\mu-\rho}}}$.
Suppose that $\sigma=\mu$ and let $K=\GEN{b^{2^{\mu-\rho}}}$.
Then $(H,K)$ satisfies the conditions of \Cref{SSPMetabelian}, and by \Cref{SSPAlg}, we have that $\Q Ge(G,H,K)$ has degree $[G:H]=2^{\mu-\rho}$ and center $F$.

Otherwise, by condition \ref{epsilon-1} in \Cref{ClasiMeta} we have $\sigma=\mu-1$.
By means of contradiction suppose that $\Q G$ has a simple component $A$ of degree $2^{\mu-\rho}$ and center $F$.
Then $H=\GEN{a,b^{2^{\mu-\rho}}}$.
As $H$ is maximal abelian subgroup of $G$ with $G/H$ abelian, by \Cref{SSPMetabelian}, we have $A=\Q Ge(G,H_1,K)$ for subgroups $H_1$ and $K$ satisfying the conditions of \Cref{SSPMetabelian} and $H_1\supseteq H$.
However, by \Cref{SSPAlg}, $[G:H]=2^{\mu-\rho}=\Deg(A)=[G:H_1]$ and hence $H_1=H$.
As $H/K$ is cyclic, either $H=\GEN{a,K}$ or $H=\GEN{b^{2^{\mu-\rho}},K}$.
In the second case $N_G(K)/K$ is abelian and by \Cref{SSPAlg}, the center $F$ of $A$ is a cyclotomic extension of $\Q$, which is not the case.
Therefore $H=\GEN{a,K}$.
In particular $[H:K]\le |a|=2^{\mu}$.
If $a^{2^{\mu-1}}\in K$ then $\GEN{a^{2^{\mu-1}}}=\GEN{a,b^{2^{\mu-\rho-1}}}'\subseteq K\unlhd \GEN{a,b^{2^{\mu-\rho-1}}}$ and $\GEN{a,b^{2^{\mu-\rho-1}}}$ contains properly $H$, in contradiction with the assumption that $(H,K)$ satisfy condition (1) of \Cref{SSPMetabelian}.
Therefore $K\cap \GEN{a}=1$ and hence $[H:K]\ge |a|=2^{\mu}$.
So $[H:K]=2^{\mu}$. As $N_G(K)/H\cong \Gal(\Q(\zeta_{[H:K]})/F)$, we have $[N_G(K):H]=[\Q(\zeta_{[H:K]}):F]=2^{\mu-\rho}=[G:H]$ and hence $G=N_G(K)$, i.e. $K\unlhd G$. As $K\cap G'=1$ it follows that $K\subseteq Z(G)=\GEN{a^{2^{\mu-1}},b^{2^{\mu-\rho}}}=\GEN{b^{2^{\mu-\rho}}}$.
Finally, the assumption $\mu<\nu+\rho$ implies that $H$ contains $\GEN{a}$ properly. Therefore $|H|>2^{\mu}$ and hence $K$ is a non-trivial subgroup of the cyclic subgroup $\GEN{b^{2^{\mu-\rho}}}$. Thus $K$  contains the unique element of order $2$ of $Z(G)$, namely $a^{2^{\mu-1}}$, a contradiction.
\end{proof}

We are ready to prove the main result of this section.
\begin{theorem}\label{Isop}
Let $p$ be prime integer. If $G_1$ and $G_2$ are finite metacyclic $p$-groups and $\Q G_1\cong \Q G_2$ then $G_1\cong G_2$.
\end{theorem}

\begin{proof}
Suppose that $\Q G_1\cong \Q G_2$.
By \Cref{ClasiMeta}, we have
$G_i\cong \mathcal{P}_{p,\mu_i,\nu,\sigma_i,\rho_i,\epsilon_i}$
with each list $\mu_i,\nu_i,\sigma_i,\rho_i,\epsilon_i$ satisfying conditions \ref{rhomu}-\ref{epsilon-1}.
We will prove that $(\mu_1,\nu_1,\sigma_1,\rho_1,\epsilon_1)=(\mu_2,\nu_2,\sigma_2,\rho_2,\epsilon_2)$.

First of all $p^{\mu_1+\nu_1}=|G_1|=|G_2|=p^{\mu_2+\nu_2}$ and hence
$\mu_1+\nu_1=\mu_2+\nu_2$.
Moreover, by \Cref{Abelianizado} we have $G_1/G'_1\cong G_2/G'_2$ and from conditions \ref{epsilon1} and \ref{epsilon-1} it follows that
$$G_i/G'_i \cong \begin{cases} C_{p^{\rho_i}}\times C_{p^{\nu_i}}, & \text{if } \epsilon_i=1, \\
C_2\times C_{2^{\nu_i}}, & \text{if } \epsilon_i=-1\end{cases}$$

Suppose that $\epsilon_1=1$ and $\epsilon_2=-1$.
Then $C_{2^{\rho_1}}\times C_{2^{\nu_i}}\cong C_2\times C_{2^{\nu_2}}$, by \Cref{Abelianizado}, and by conditions \ref{epsilon1} and \ref{epsilon-1} we have  $p=2$,  $\rho_1\le \nu_1$, $2\le \rho_2$ and $1\le \nu_2$. Therefore $\rho_1=1$ and hence $\mu_1=1$ by condition \ref{rhomu}.
This implies that $G_1$ is abelian but $G_2$ is not abelian, in contradiction with $\Q G_1\cong \Q G_2$.
This proves that $\epsilon_1=\epsilon_2$, which we denote $\epsilon$ from now on.

Moreover, if $\epsilon=1$ then $C_{p^{\rho_1}}\times C_{p^{\nu_1}}\cong C_{p^{\rho_2}}\times C_{p^{\nu_2}}$ with $\rho_i\le \nu_i$, and if $\epsilon=-1$ then $C_2\times C_{2^{\nu_1}}\cong C_2\times C_{2^{\nu_2}}$ and $1\le \nu_1,\nu_2$. Thus, in both cases $\nu_1=\nu_2$, and hence $\mu_1=\mu_2$.
From now on we set $\mu=\mu_i$ and $\nu=\nu_i$.
Suppose that $\epsilon=1$ then $C_{p^{\rho_1}}\times C_{p^{\nu_1}}\cong C_{p^{\rho_2}}\times C_{p^{\nu_2}}$ and hence $\rho_1=\rho_2$, which we denote $\rho$.
Moreover, by Artin's Theorem (\Cref{Artin}), the number of Wedderburn components of $\Q G_i$ is the number of conjugacy classes of subgroups of $G_i$.
Therefore if $A_{\sigma_1}$ and $A_{\sigma_2}$ are as defined in
\Cref{CCCOdd} then we have $A_{\sigma_1}=A_{\sigma_2}$.
Let
	$$B_{\sigma_i}=2p^{\mu-\rho}(p-1)A_i =
	-2p^{\sigma_i}+\sigma_i p^{\mu-1}(p-1)(2+(p-1)(1+2\nu-\sigma_i)).$$
Then $B_{\sigma_1}=B_{\sigma_2}$.
By means of contradiction, assume without loss of generality that $\sigma_1<\sigma_2$.
By condition \ref{epsilon1} we have $\sigma_1<\sigma_2\le \mu\le \nu+\rho$. If $\sigma_1<\mu-1$ then $\min(\sigma_2,\mu-1)\le v_p(B_{\sigma_2})=v_p(B_{\sigma_1})=\sigma_1<\mu-1$, which contradicts the assumption $\sigma_2>\sigma_1$. Therefore, $\mu-1\le \sigma_1<\sigma_2\le \min(\mu,\nu)$, i.e. $\sigma_1=\mu-1$ and $\sigma_2=\mu\le \nu$.
Then
\begin{eqnarray*}
	0&=&B_\mu-B_{\mu-1} \\
	&=&
	-2p^\mu+\mu p^{\mu-1}(p-1)(2+(p-1)(2\nu+1-\mu)) \\
	&&+2p^{\mu-1}-(\mu-1) p^{\mu-1}(p-1)(2+(p-1)(2\nu+1-(\mu-1))) \\
	&=&
	p^{\mu-1}(p-1)[ -2+\mu (2+(p-1)(2\nu+1-\mu)) -(\mu-1) (2+(p-1)(2\nu+2-\mu))] \\
	&=&
	p^{\mu-1}(p-1)[ -2+2\mu-2(\mu-1)+\mu (p-1)(2\nu+1-\mu)
	-(\mu-1) (p-1)(2\nu+2-\mu)] \\
	&=&
	p^{\mu-1}(p-1)[ \mu (p-1)(2\nu+1-\mu)
	-\mu (p-1)(2\nu+2-\mu)+(p-1)(2\nu+2-\mu)]
	\\
	&=&
	2p^{\mu-1}(p-1)^2(\nu+1-\mu) >0,
\end{eqnarray*}
which is the desired contradiction.

Suppose now that $\epsilon=-1$.
We first prove that $\rho_1=\rho_2$.
By means of contradiction suppose that $\rho_1<\rho_2$.
It is well known that the dimension over $\Q$ of the center of $\Q G_i$ is the number of conjugacy classes of $G_i$. Then, by \Cref{CCN} we have
	$$2^{\rho_1}(3\cdot 2^{\nu-1}-2^{\nu+\rho_1-\mu})=2^{\rho_2}(3\cdot 2^{\nu-1}-2^{\nu+\rho_2-\mu})$$
If $\rho_2<\mu-1$ then
$$2\rho_2+\nu-\mu=v_2(2^{\rho_2}(3\cdot 2^{\nu-1}-2^{\nu+\rho_2-\mu})=
v_2(2^{\rho_1}(3\cdot 2^{\nu-1}-2^{\nu+\rho_1-\mu}))=2\rho_1+\nu-\mu,$$
which contradicts the assumption $\rho_1<\rho_2$.
Therefore $\rho_2\ge \mu-1$.
If $\rho_1<\mu-1$ then using that $\mu\ge 2$, by condition \ref{epsilon-1}, we have
$$\rho_2+\nu-1\le v_2(2^{\rho_2}(3\cdot 2^{\nu-1}-2^{\nu+\rho_2-\mu})=
v_2(2^{\rho_1}(3\cdot 2^{\nu-1}-2^{\nu+\rho_1-\mu}))=2\rho_1+\nu-\mu<\rho_1+\nu-1,$$
again in contradiction with the assumption $\rho_1<\rho_2$.
Therefore $\rho_1=\mu-1$ and $\rho_2=\mu$ and hence $\mu\ge 3$, by condition \ref{rhomu}.
If $\sigma_2=\mu$ then, by \Cref{SpecialcaseNr}, $\Q G_2$ has a simple component with center isomorphic to $\Q(\zeta_{2^\mu}+\zeta_{2^\mu})$ while $\Q G_1$ does not.
Therefore $\sigma_2=\mu-1$. This implies that $\nu=1$, by condition \ref{epsilon-1}.
Therefore $G_2$ is the quaternion group of order $2^{\mu+1}$.
If $\sigma_1=\mu$ then $G_1$ is the dihedral group of order $2^{\mu+1}$.
Otherwise $\sigma_1=\mu-1$ and if $b_1=ba$ then $b_1^2=1$ so that $G_1$ is the semidihedral group $\GEN{a,b_1\mid a^{2^{\mu-1}} = b_1^2 = 1, a^{b_1}=a^{-1+2^{\mu-1}}}$.
Looking at the Wedderburn decomposition of the rational group algebras of dihedral, semidihedral groups and quaternion group in \cite[19.4.1]{JespersdelRioGRG1} we deduce that $\Q G_2$ has a simple component isomorphic to the quaternion algebra  $\HQ(\Q(\zeta_{2^\mu}+\zeta_{2^\mu}))$, which is a non-commutative division algebra, while $\Q G_1$ does not have any Wedderburn component which is a non-commutative division algebra. This yields the desired contradiction in this case.

So we can set $\rho=\rho_1=\rho_2$ and it remains to prove that $\sigma_1=\sigma_2$. Otherwise, we may assume that $\sigma_1=\mu-1$ and $\sigma_2=\mu<\nu+\rho$, by condition \ref{epsilon-1}.
If $\rho<\mu$ then we obtain a contradiction with \Cref{GroupsNscalculus}. Thus $\rho=\mu$. If $\mu\ge 3$ then the  contradiction follows from \Cref{SpecialcaseNr}.
Thus $\mu=2$ but then $G_1$ is the quaternion group of order $8$ and $G_2$ is the dihedral group of order $8$ and again $\Q G_1$ has Wedderburn component which is a non-commutative division algebra but $\Q G_2$ does not, yielding to the final contradiction.
\end{proof}

\section{The Isomorphism Problem for finite metacyclic nilpotent groups}\label{SectionNilpotent}

Given a finite group $G$ we say that a Wedderburn component of $\Q G$ is a \emph{$p$-component} if its degree is a power of $p$ and its center embeds in $\Q(\zeta_{p^n})$ for some non-negative integer $n$.

\begin{lemma}\label{pComponente}
	Let $G$ be a finite group and $(L,K)$ a strong Shoda pair of $G$. Then $\Q Ge(G,L,K)$ is a $p$-component if and only if $[G:L]$ is a power of $p$ and $[L:K]_{p'}\in \{1,2\}$.
\end{lemma}

\begin{proof}
	The reverse implication is a direct consequence of \Cref{SSPAlg}.
	Conversely, set $A=\Q Ge(G,L,K)$ and suppose that $A$ is a $p$-component.
	Let $d=[G:L]$ and $c=[L:K]$. As $d$ is the degree of $A$, then it is a power of $p$.
	Moreover the center of $A$ is isomorphic to the Galois correspondent $F_{G,L,K}=\Q(\zeta_{c})^{\Imagen(\alpha)}$ of a subgroup of $\Gal(\Q(\zeta_c)/\Q)$ isomorphic to $N_G(K)/L$ (\Cref{SSPRemark}).
	The assumption implies that $F\subseteq \Q(\zeta_{c_p})$.
	As $[N_G(K):L]$ is a power of $p$, then so is $[\Q(\zeta_c):F_{G,L,K}]$ and hence $\varphi(c_{p'})=[\Q(\zeta_c):\Q(\zeta_{c_p})]$ is a power of $p$.
	Then $c_{p'}$ is either $1$ or $2$.
\end{proof}

If $G$ is a finite group then we use the notation
\begin{equation*}\label{piDef}
\pi_G=\{p\in \pi(G) : G \text{ has a normal Hall } p'-\text{subgroup}\} \qand
\pi'_G=\pi(G)\setminus \pi_G.
\end{equation*}

\begin{remark}\label{Minimalinpi}
	If $G$ is metacyclic and $p$ is the smallest prime dividing $|G|$ then $p\in \pi_G$. In particular, $2\not\in \pi'_G$.
\end{remark}

\begin{proof}
Let $\pi=\pi_G$ and $\pi'=\pi'_G$.
If $p\in \pi'$ then by \cite[Lemma~3.1]{GarciadelRio2022}, $G'_p$ has a non-central element $h$ of order $p$. Therefore $G$ contains an element $g$ such that $[g,h]\ne 1$ and we may assume that $|g|$ is a power of a prime $q$. Then $\Aut(\GEN{h})$ has an element of order $q$. As $\Aut(\GEN{h})$ has order $p-1$ it follows that $q\mid p-1$ and in particular $q>p$. Thus $p$ is not the smallest prime dividing $|G|$.
\end{proof}

\begin{lemma}\label{piDetermined}
If $G$ and $H$ are metacyclic groups with $\Q G\cong \Q H$ then $\pi'_G=\pi'_H$ and $\pi_G=\pi_H$.
\end{lemma}

\begin{proof}
Let $\pi=\pi_G$ and $\pi'=\pi'_G$.
We claim that $\pi'=\{p\in \pi(G') : (G/G')_p \text{ is cyclic}\}$.
Let $A=\GEN{a}\unlhd G$ and $B=\GEN{b}\le G$ with $G=AB$.
By \cite[Lemma~3.1]{GarciadelRio2022}, $\GEN{a_p,b_p}$ is a Sylow $p$-subgroup of $G$, $A_{\pi'}=G'_{\pi'}$ and $G=A_{\pi'}\rtimes \left(B_{\pi'}\times \prod_{q\in \pi} A_qB_q\right)$. Therefore, if $p\in \pi'$ then $(G/G')_p$ is cyclic.
If $p\in \pi'\setminus \pi(G')$ then $A_{\pi'}\rtimes \left(B_{\pi'\setminus \{p\}}\times \prod_{q\in \pi} A_qB_q\right)$ is a normal Hall $p'$-subgroup of $G$ and hence $p\in \pi$, a contradiction.
This proves that $\pi'\subseteq \{p\in \pi(G') : (G/G')_p \text{ is cyclic}\}$.
Conversely, if $p\in \pi$ then $[b_{p'},a_p]=1$ and therefore
$G'_p=\GEN{a_p,b_p}'$.
Then $(G/G')_p\cong \GEN{a_p,b_p}/\GEN{a_p,b_p}'$.
Therefore, if $(G/G')_p$ is cyclic then so is $\GEN{a_p,b_p}$ by the Burnside Basis Theorem.
In that case $1=\GEN{a_p,b_p}=G'_p$, i.e. $p\not\in \pi(G')$.
This finishes the proof of the claim.

By \Cref{Abelianizado}, the assumption implies that $G/G'\cong H/H'$ and hence $|G'|=|H'|$.
Then $G'\cong H'$ as both $G'$ and $H'$ are cyclic.
Then, using the claim for $G$ and $H$ we deduce that $\pi'_G=\{p\in \pi(G'):(G/G')_p \text{ is cyclic}\} =\{p\in \pi(H'):(H/H')_p \text{ is cyclic}\} =\pi'_H$ and $\pi_G=\pi(|G|)\setminus \pi'_G=\pi(|H|)\setminus \pi'_H=\pi_H$.
\end{proof}

In the remainder of the paper if $G$ is a group and $p$ is a prime then $G_p$ denotes a Sylow subgroup of $G$ and $G_{p'}$ a Hall $p'$-subgroup of $G$.

\begin{lemma}\label{piComponent}
If $G$ is metacyclic and $p\in \pi_G$ then the sum of the $p$-components of $\Q G$ is isomorphic to a direct product of $k$ copies of $\Q G_p$, where
	$$k=\begin{cases} 1,&  \text{if } p=2; \\ [G_2:G'_2G_2^2], & \text{otherwise}.\end{cases}$$
\end{lemma}

\begin{proof}
Let $\pi=\pi_G$ and $\pi'=\pi'_G$ and suppose that $p\in \pi$.
By \Cref{Minimalinpi}, $2\not\in \pi'$  and hence $G$ has a normal Hall $\{2,p\}'$-subgroup $N$.
Let $e$ be a primitive central idempotent such that $\Q Ge$ is a $p$-component of $G$. Then $e=e(G,L,K)$ for some strong Shoda pair $(L,K)$ of $G$ and, by \Cref{pComponente} $[G:K]$ is either a power of $p$ or $2$ times a power of $p$.
In particular $N\subseteq K$.
Then $\widehat{N}\widehat{M}=\widehat{M}$ for every subgroup $M$ containing $K$ and as $N$ is normal in $G$ we also have $\widehat{N}\widehat{M}^g=0$ for every $g\in G$. This implies that $\widehat{N}e=e$.
This proves that every $p$-component of $\Q G$ is contained in $\Q G\widehat{N}$. Therefore $\Q G \widehat{N}=A\oplus B$ where $A$ is the sum of the $p$-components of $\Q G$, and $B$ is the sum of the Wedderburn components of $\Q G\widehat{N}$ which are not $p$-components. We want to prove that $\Q(G_p)^k\cong A$.

Suppose first that $p=2$. Therefore $N=G_{2'}$ and hence $G/N\cong G_2$. Thus $G/N$ is a $2$-group and hence every Wedderburn component of $\Q(G/N)$, and $\Q G \widehat{N}$, is a $p$-component.  Therefore $\Q(G_2)\cong \Q G \widehat{N}=A$, as desired.

Suppose that $p\ne 2$. Then $G/N=U_2\times U_p'$ with $U_2=G_{p'}/N\cong G_2$, and $U_p=G_{2'}/N\cong G_p$.
Let $F_2=U_2'U_2^2$, the Frattini subgroup of $U_2$.
Then $F_2=L/N$ for some subgroup $L$ of $G_{p'}$ and by \Cref{pComponente} it follows that $L\subseteq K$ and the argument in the first paragraph shows that every $p$-component of $\Q G$ is contained in $\Q G \widehat{L}$.
Thus $\Q G \widehat{L}=A\oplus C$ where $C$ is the sum of the Wedderburn components of $\Q G\widehat{L}$ which are not $p$-components.
Moreover, $G/L\cong U_p\times E$ for $E$ an elementary abelian $2$-group of order $k$.
Then $\Q E\cong \Q^k$ and hence $\Q G\widehat{L} \cong \Q(G/L)\cong (\Q U_p)^k$.
Moreover, as $U_p$ is a $p$-group, every Wedderburn component of $\Q U_p$ is a $p$-component. In other words, $C=0$ and hence $A\cong  (\Q U_p)^k = (\Q G_p)^k$, as desired.
\end{proof}

\begin{lemma}\label{piSylowAlgebras}
Let $G$ and $H$ be finite metacyclic groups with $\Q G\cong \Q H$, let $p\in \pi_G$ and let $G_p$ and $H_p$ be Sylow subgroups of $G$ and $H$ respectively.
Then $\Q G_p\cong \Q H_p$.
\end{lemma}

\begin{proof}
Let $\pi=\pi_G$ and $\pi'=\pi'_G$ and let $k$ be as in \Cref{piComponent}. As $2\not\in \pi'$, by \Cref{Minimalinpi},  and $G/G_{\pi'}$ is nilpotent, it follows that $G_2/G'_2G_2^2$ is isomorphic to the Sylow 2-subgroup of the quotient $G/G'$ by its Frattini subgroup. Since $G/G'\cong H/H'$, the value of $k$ is the same whether it is computed for $G$ or $H$.
Let $A_G$ and $A_H$ be the sum of the Wedderburn $p$-components of $\Q G$ and $\Q H$. Since $\Q G\cong \Q H$ then $A_G\cong A_H$. By \Cref{piComponent},
$(\Q G_p)^k\cong A_G\cong A_H \cong (\Q H_p)^k$ an therefore $\Q G_p\cong \Q H_p$.
\end{proof}

We are ready to proof our main result:

\begin{proofof}\emph{\Cref{MainpiIgual}}.
By \Cref{piDetermined} we have $\pi_G=\pi_H$ and from now on we denote the latter by $\pi$.
Then the Hall $\pi$-subgroups of $G$ and $H$ are nilpotent and hence it is enough to prove that if $p\in \pi$ then the Sylow $p$-subgroups $G_p$ of $G$ and $H_p$ of $H$ are isomorphic. However, $\Q G_p\cong \Q H_p$, by \Cref{piSylowAlgebras},  and hence $G_p\cong H_p$, by \Cref{Isop}.
\end{proofof}

If $G$ is nilpotent then $\pi'_G=\emptyset$ and hence \Cref{Main} follows directly from \Cref{MainpiIgual}.


\bibliographystyle{amsalpha}
\bibliography{ReferencesMSC}
\end{document}